\documentclass[11pt]{article}
\usepackage[all]{xy}
\usepackage{amsmath}
\usepackage{amsfonts}
\usepackage{amssymb}
\usepackage{amscd}
\usepackage{amsthm}
\usepackage{latexsym}
\usepackage{amsbsy}
\usepackage{xcolor,enumerate}
\usepackage[normalem]{ulem}

\def\0D{\Delta^{(0)}}
\def\1D{\Delta^{(1)}}

\newcommand{\Hc}{\mathcal{H}}
\newcommand{\Cc}{\mathcal{C}}

\newcommand{\Mc}{\mathcal{M}}
\newcommand{\Zc}{\mathcal{Z}}
\newcommand{\Fc}{\mathcal{F}}
\newcommand{\Oc}{\mathcal{O}}
\newcommand{\hmod}{{_H\mathcal{M}}}
\newcommand{\md}{\mathbb{M}\text{d}}
\newcommand{\Dc}{\mathcal{D}}
\newcommand{\Nc}{\mathcal{N}}

\newcommand{\Ec}{\mathcal{E}}

\newtheorem{theorem}{Theorem}[section]
\newtheorem{remark}[theorem]{Remark}
\newtheorem{proposition}[theorem]{Proposition}
\newtheorem{lemma}[theorem]{Lemma}

\newtheorem{example}[theorem]{Example}
\newtheorem{definition}[theorem]{Definition}

\def\build#1_#2^#3{\mathrel{\mathop{\kern 0pt#1}\limits_{#2}^{#3}}}

\newcommand{\ps}[1]{~\hspace{-3pt}^{^{\left(#1\right)}}}

\newcommand{\vect}{\text{Vec}}
\newcommand{\id}{1}
\newcommand{\bimod}{\mathbb{B}\text{md}}

\newcommand{\Hm}{\mathcal{H}om}

\newcommand{\la}{\triangleright}
\newcommand{\ra}{\triangleleft}
\newcommand{\Bc}{\mathcal{B}}
\newcommand{\dF}{{^*F}}

\numberwithin{equation}{section}
\def\ot{\otimes}
\def\part{\partial}

\def\ot{\otimes}

\def\Hom{\mathop{\rm Hom}\nolimits}

\def\build#1_#2^#3{\mathrel{
\mathop{\kern 0pt#1}\limits_{#2}^{#3}}}

\numberwithin{equation}{section}
\newcommand{\comment}[1]{\relax}

\begin{document}
%\Large
\title{Some invariance properties of cyclic cohomology with coefficients.}
\author {Ilya Shapiro}
\date{
}
\maketitle
\begin{abstract}
In this paper, we further explore the conceptual approach to cyclic cohomology with coefficients initiated in \cite{hks}.  In particular we give a derived version of the definition with better invariance properties.  We show that the new definition agrees with the old under certain conditions and we prove, for the new definition, both its Morita invariance and its $2$-Morita invariance, under suitable interpretations of these terms. More generally, we prove a version of Shapiro's lemma for cyclic cohomology with coefficients.

\end{abstract}

\medskip

{\it 2010 Mathematics Subject Classification.} Monoidal category (18D10), abelian and additive category (18E05), cyclic homology (19D55), Hopf algebras	(16T05).

%%%%%%%%%%%%%%%%%%%%%%%%%%%%%%%%%%%%%%%%%%%%%%%%%%%%%%%%%%%%%%%%%%%%%%%%%%%%%%%%%%%%%%%%%%%%%%%%%%%%%%%%%%%%%%%%%%%%%%%%%%%%%%%%%%%%%%%%%%%%%%%%%%%%%%%%%%%%%%%%%%%%%%%%%%%%%%%%%%%%%%%%%%%%%%%%%%%%%%%%%%%%%%%%%%%%%%%%%%%%%%%%%%%%%%%%%%%%%%%%%%%%%%%%%
\section{Introduction}
Since its introduction independently by Boris Tsygan and Alain Connes in the 1980s cyclic (co)homology has branched out in many directions and is currently the subject of vigorous research.  The branch that we contribute to in this paper follows the development of an ``equivariant" version of cyclic cohomology that began with the work of Connes-Moscovici and was generalized into Hopf-cyclic cohomology of Hajac-Khalkhali-Rangipour-Sommerh\"{a}user.  We further generalize this development in order to better understand their work.  Our methods use the categorical, $2$-categorical and even $3$-categorical way of thinking and so we would like to point out that such approaches to cyclic theory were made before.  In particular we may cite \cite{catdual, catex, catcoeff, catadj, catmackey} as some examples where this categorical approach resulted in a better understanding and various generalizations.  What we pursue here was motivated in part by  the author's prior joint work \cite{hks} that considered cyclic cohomology in monoidal categories from the point of view of a categorical notion of a trace.  The main point of the present paper is Theorem \ref{sl}; it relates cyclic cohomologies between possibly very different looking monoidal categories.

More precisely, it was shown in  \cite{hks} how to assign to a unital associative algebra $A$ in a monoidal category $\Cc$ equipped with a $2$-contratrace a cocyclic object.  Furthermore, it turned out that in the case of $\Cc=\hmod$, the category of left modules over a Hopf algebra $H$, the method recovered the Hopf cyclic cohomology, type $A$, with coefficients; other types were explored as well.  In this paper we consider the situation of a unital associative algebra $A$ in a monoidal category $\Cc$.  This recovers not only the type $A$ theory with $\Cc=\hmod$, but also the type $B$ theory, here $\Cc$ is the category of right $H$-comodules $\Mc^H$.  For purposes of brevity we do not deal with coalgebras, except  briefly in Section \ref{coalgebrasection} where we need them to more tightly tie together the categories of Hopf modules and Hopf comodules.

Whereas \cite{hks} was concerned with giving a general definition, for monoidal categories, of cyclic cohomology with coefficients that reduces to the Hopf-cyclic cohomology for Hopf algebras, here we pursue a different goal.   We change the definition of cyclic cohomology in monoidal categories to a more homologically pleasing derived version.  This change does not impact the usual cyclic cohomology, nor more generally cohomology in monoidal categories with a projective identity, see Proposition \ref{oldvsnew}.  It does change the existing Hopf-cyclic cohomology in some non semi-simple examples of Hopf algebras (see Remark \ref{diff}).  Some of the arguments for the new definition are Lemma \ref{derint} and of course the main result: Theorem \ref{sl}, that fails for the old definition in general.  The new definition is based on a (different from \cite{hks}) construction of precocyclic objects in Section \ref{precocyclic}.   The  price of invariance of our new definition is that we no longer construct cocyclic objects; we deal instead with precocyclic objects.  However precocyclic structure is sufficient for both Hochschild and cyclic cohomology \cite{loday}.

The setting is that of abelian monoidal categories, such as those of modules or comodules over Hopf algebras, or as a very different flavored example: left modules over a commutative algebra.  The original setting of cyclic cohomology was the category of vector spaces $\vect$, and so no non-trivial coefficients were possible.  Hopf cyclic cohomology with coefficients \cite{cm, hkrs} takes place in the categories of modules or comodules over Hopf algebras, where the coefficients are not only possible, they are required.  More general situations do fit into the framework, for example modules over weak Hopf algebras; the latter is dealt with tangentially in this text. The example that started this line of investigation, and we feel explains both the ``equivariant" nature of Hopf-cyclic cohomology and its ability to account for local system coefficients,  concerns only the category of representations of a finite group $G$, it can be found in Section \ref{derhamsection}.

When we talk about duality we mean it in the sense of \cite{EGNO}; with the important distinction that we must assume that all of our module categories arise from algebras.  Thus if $\Cc$ is a monoidal category, the only types of right $\Cc$-module categories that we consider are of the form: left $A$-modules in $\Cc$ for $A$ a unital associative algebra in $\Cc$. Of course this means that the only duals are of the form: $A$-bimodules in $\Cc$, but thinking of them as duals makes explicit the Morita invariance of all that we consider.  Note that if one considers non-unital associative algebras as a special case of unital associative algebras, namely those of the form $\id\oplus A$ then the results of this paper apply to them as well.

Thinking about duality led us to the realization that cyclic cohomology should in fact be invariant under $2$-Morita equivalence, since it is duality that supplies the examples.  More precisely, under conjectured Morita invariance we expect to have an isomorphism on cyclic cohomology, with a fixed coefficient object, between two unital associative algebras in $\Cc$ that are Morita equivalent (in particular, their module categories are equivalent).  This is proven in Proposition \ref{cyclicismoritainvar}.  But one furthermore conjectures that if two monoidal categories $\Cc$ and $\Dc$ are $2$-Morita equivalent, in particular their right module $2$-categories are equivalent (modulo correctly chosen definitions), then there should be a way to transfer algebras and coefficient objects between the two categories that would result in isomorphic cyclic cohomologies.  This is explained in Section \ref{2moritainvariance}.  Even more generally, we have a type of adjunction on cyclic cohomology when two monoidal categories are not $2$-Morita equivalent, rather there is a suitable $2$-functor between their module $2$-categories or viewed differently a $1$-morphism from one to the other in a suitable $3$-category, see Section \ref{adjunctionsection} and Theorem \ref{sl} in particular.

The correct notion of the $2$-category of right $\Cc$-modules is $\bimod_\Cc$ whose objects are unital associative algebras in $\Cc$, $1$-morphisms are bimodules, and $2$-morphisms are bimodule homomorphisms.  While in general not every right $\Cc$-module category arises from algebras, this is a good replacement for that notion.  Thinking about $\bimod_\Cc$ in these terms is helpful in explaining some of the constructions encountered in this paper.

We reiterate that the goal of this paper is not only to explore this general setting so that it can accommodate notions more general than Hopf algebras, but also to better understand the already well studied objects and as yet unexplored relationships between them.  For example, in Section \ref{adjactionsection} we demystify the notion of contramodule coefficients, while in Section \ref{natadj} we explore the close relationship between the type $A$ and type $B$ Hopf-cyclic cohomology theories.

Some technical details: we do not assume that our monoidal categories are strict, but we do suppress the explicit writing out of the associativity constraints.  All categories are abelian, furthermore they are $k$-linear for a field $k$ of characteristic $0$ and so are all the functors that we consider. The $\ot$ is right exact in both variables and distributes over $\oplus$.  Monoidal functors are strongly monoidal.  Semisimplicity is not assumed, in fact it would make everything uninteresting, unless we quickly give it up.  We do assume that $\Cc$ has enough projective objects; this is because we are dealing with algebras, coalgebras would flip all the requirements.

\bigskip

\textbf{Acknowledgments}: The author wishes to thank Masoud Khalkhali for stimulating questions, encouragement, and pointing out many useful references. Furthermore, thanks are due to Piotr Hajac, Ivan Kobyzev, Georgy Sharygin, and Serkan S\"{u}tl\"{u} for illuminating discussions. This research was supported in part by the NSERC Discovery Grant number 406709.

\bigskip

\section{A review of symmetric $2$-contratraces}\label{review}
Let us recall one of the key notions of \cite{hks}.  The exposition below differs from the one in \cite{hks} in its focus on the notion of a $2$-contratrace instead of that of a $2$-trace.

For a monoidal category $\Cc$ let $$\Cc^*=Fun(\Cc^{op},\vect)^{op}$$ and note that $\Cc^*$ is a $\Cc$-bimodule category.  More precisely, $$c\cdot \Fc\cdot d(-)=\Fc(d\ot -\ot c).$$ The $(-)^{op}$ denotes the opposite category of the underlying abelian category; the first $op$ ensures that our functors are contravariant, so that we get cohomology from algebras, while the second $op$ makes the arrows align properly with respect to the actions.

Consider now the center of the bimodule category $\Cc^*$ which we denote by $\Zc_\Cc\Cc^*$.  Roughly speaking it consists of contravariant functors $\Fc$ such that they come equipped with bifunctorial isomorphisms $$\tau_{c,d}:\Fc(c\ot d)\rightarrow \Fc(d\ot c).$$  We called such objects $2$-contratraces.  Furthermore, we say that $\Fc$ is a symmetric $2$-contratrace if $$\tau_{d,c}\circ\tau_{c,d}=Id_{\Fc(c\ot d)}$$ and we denote that subcategory of the center by $\Zc'_\Cc\Cc^*$, called the restricted center.

Recall that if $\Fc$ is a symmetric $2$-contratrace then for a unital associative algebra $A$ in $\Cc$ we have that $$C^n=\Fc(A^{\ot n+1})$$ is a cocyclic object with $\tau$ obtained from $\tau_{A,A^{\ot n}}$, cofaces from the multiplication in $A$, and codegeneracies from the inclusion of the unit in $A$.  The following is the definition used in \cite{hks}.

\begin{definition}
Let $A$ be a unital associative algebra in $\Cc$ and $\Fc$ a symmetric $2$-contratrace, then if $C^n=\Fc(A^{\ot n+1})$ is the cocyclic object described above, we defined $$HC^n_{old}(A,\Fc)=HC^n(C^\bullet).$$
\end{definition}

In \cite{hks} it is shown that when $\Cc=\hmod$, the category of left modules over a Hopf algebra $H$ and $M\in{_H SAYD^H}$, i.e., $M$ is a left-right stable anti-Yetter-Drinfeld module, then $$\Hom_H(M\ot -, \id)$$ is an example of a symmetric $2$-contratrace that recovers the type $A$ theory of \cite{hkrs}.  One can show that similarly, $$\Hom(-,M)^H$$ is a symmetric $2$-contratrace that recovers the type $B$ theory of \cite{hkrs} when applied to the category $\Cc=\Mc^H$ of right $H$-comodules.

\section{A construction of precocyclic objects}\label{precocyclic}
Let $\Dc$ be a monoidal category, and $\Fc$ be an element in $\Zc'_{\Dc}(\Dc^{*})$, i.e., $\Fc$ is a symmetric $2$-contratrace on $\Dc$.  Let $d:P\rightarrow \id$ be given.

\begin{definition}\label{precocyclicstructure}
Let $$C^n=C^n(d,\Fc)=\Fc(P^{\ot n+1}).$$

For $i=0, \cdots, n+1$ let $\delta_i': P^{\ot n+2}\rightarrow P^{\ot n+1}$ be given by $$\delta_i'=Id^{\ot i}_P\ot d\ot Id_P^{\ot n+1-i},$$ and let $$\delta_i=\Fc(\delta_i'): C^n\rightarrow C^{n+1}.$$

Define $\tau_n\in Aut(C^n)$ via $$\tau_n=\tau_{P,P^{\ot n}}.$$
\end{definition}

\begin{remark}\rm{
Note that in the above definitions of cofaces, there is no ``flipover coface", i.e., a special $\delta_n$ that is not like the others in the sense that its definition utilizes $\tau$ itself.  The reason for this will become clear later.
}\end{remark}

\begin{proposition}
The structure on $C^n$ defined above satisfies the relations:
\begin{align*}\delta_j\delta_i&=\delta_i\delta_{j-1},\, 0\leq i<j\leq n+1\\\tau_{n+1}\delta_i&=\delta_{i-1}\tau_n,\, i\geq 1\\ \tau_{n+1}\delta_0&=\delta_{n+1},\\\tau^{n+1}_n&=Id,\end{align*}
and thus makes $C^n$ into a precocyclic object.
\end{proposition}

\begin{proof}
Let $i<j$, then \begin{align*}\delta_i'\delta_j'&=Id^{\ot i}_P\ot d\ot Id_P^{\ot n-i}\circ Id^{\ot j}_P\ot d\ot Id_P^{\ot n+1-j}\\&=Id^{\ot i}_P\ot d\ot Id_P^{\ot j-i-1}\ot d\ot Id_P^{\ot n+1-j}\\&=Id^{\ot j-1}_P\ot d\ot Id_P^{\ot n+1-j}\circ Id^{\ot i}_P\ot d\ot Id_P^{\ot n+1-i}\\&=\delta_{j-1}'\delta_i'.\end{align*}  So after applying $\Fc$ we deduce the satisfaction of the axioms for a precosimpicial structure.

On the other hand $\tau$ is obtained via a bifunctorial isomorphism $\tau_{c,d}:\Fc(c\ot d)\rightarrow\Fc(d\ot c)$, and so we have commutative squares: for $i\geq 1$ $$\xymatrix{
\Fc(P\ot P^{\ot n+1})\ar[r]^{\tau_{n+1}}&\Fc(P^{\ot n+1}\ot P)\\
\Fc(P\ot P^{\ot n})\ar[r]^{\tau_{n}}\ar[u]^{\delta_i=\Fc(Id_P\ot\delta_{i-1}')}&\Fc(P^{\ot n}\ot P)\ar[u]_{\Fc(\delta_{i-1}'\ot Id_P)=\delta_{i-1}}
}$$ and for $i=0$ $$\xymatrix{
\Fc(P\ot P^{\ot n+1})\ar[r]^{\tau_{n+1}}&\Fc(P^{\ot n+1}\ot P)\\
\Fc(\id\ot P^{\ot n+1})\ar[r]^{Id}\ar[u]^{\delta_0=\Fc(d\ot Id_P^{\ot n+1})}&\Fc(P^{\ot n+1}\ot \id)\ar[u]_{\Fc(Id_P^{\ot n+1}\ot d)=\delta_{n+1}}.
}$$

Furthermore, on $\Fc(P^{\ot n+1})$ we observe that $\tau_n^{n+1}=\tau_n^n\circ\tau_n=\tau_{P^{\ot n}, P}\circ\tau_{P, P^{\ot n}}= Id$.
\end{proof}

We note that given a $d_Q:Q\rightarrow \id$ and a commutative triangle $$\xymatrix{
P\ar[rr]^f\ar[rd]_d&&Q\ar[ld]^{d_Q}\\&\id&
}$$ we have an induced map of precocyclic objects $$f^*: C^n(d_Q,\Fc)\rightarrow C^n(d,\Fc).$$  In fact $\Fc$ yields a contravariant functor from $\Dc/\id$ to the category of precocyclic objects.  Compare this with \cite{hks}, which obtains from $\Fc$ a contravariant functor from the category of unital associative algebras in $\Dc$ to the category of cocyclic objects. Of course had we dropped the unitality assumption, we would have landed in precocyclic objects.
\begin{remark}\rm{
Though less direct, we may have used the more complicated methods of \cite{hks} to obtain a precocyclic object from $d:P\rightarrow\id$.  Namely, define an associative algebra structure on $P$ via $$m:P\ot P\xrightarrow{d\ot Id}\id\ot P\simeq P,$$ then $\Fc(P^{\ot n+1})$ acquires a precocyclic structure as before.  Of course the structure is the same either way.
}\end{remark}

\subsection{The dependence on the choice of pair}\label{dependence}
How much does $HC^n(C^\bullet(d,\Fc))$ depend on $(P,d)$?  When  $(P,d)$ is chosen properly not much.

\begin{definition}
We say that a pair $(P,d)$ with  $$d: P\rightarrow \id$$ in $\Dc$ is admissible if $$P^{\ot 2}\xrightarrow{d\ot Id-Id\ot d}P\xrightarrow{d}\id\rightarrow 0$$ is exact and $P$ is projective.
\end{definition}

We need a Lemma.

\begin{lemma}\label{homotopy} Suppose we have the following commutative diagram:
$$\xymatrix{
P\ar[dr]_d\ar@<-.5ex>[rr]_f \ar@<.5ex>[rr]^g& & Q\ar[dl]^d\\
&\id &
}$$ and form presimplicial objects $C_n(P)=P^{\ot n+1}$ and $C_n(Q)=Q^{\ot n+1}$ with $d_i=\delta'_i$s as before.  Then we have maps $g$ and $f$ of presimplicial objects $$\xymatrix{P^{\ot n+1}\ar@<-.5ex>[r]_{f^{\ot n+1}} \ar@<.5ex>[r]^{g^{\ot n+1}}&Q^{\ot n+1}}.$$ Now pass to the associated chain complexes ($d=\sum_{i=0}^n(-1)^i d_i$) and suppose that we have a commutative diagram:$$\xymatrix{& & Q^{\ot 2}\ar[d]^{d}\\P\ar[urr]^h\ar[rr]^{g-f}& & Q}$$  then we obtain a chain homotopy between $g$ and $f$ given by $$H=\sum_{i=0}^n (-1)^i f^{\ot i}\ot h\ot g^{\ot n-i}.$$
\end{lemma}

\begin{proof}
Let us abuse notation and pretend that our objects have elements; this simplifies notation but is not an actual assumption. Note that we have $$d(\alpha\ot\beta)=d(\alpha)\ot\beta+(-1)^{\overline{\alpha}} \alpha\ot d(\beta)$$ and $$H(\alpha\ot\beta)=H(\alpha)\ot g(\beta)+(-1)^{\overline{\alpha}} f(\alpha)\ot H(\beta).$$  We proceed by induction, on $P$ we have $dH+Hd=dh=g-f$.  Now assume that $(dH+Hd)\alpha=g\alpha-f\alpha$ and so \begin{align*}
(dH+Hd)(p\ot\alpha)&=d(hp\ot g\alpha-fp\ot H\alpha)+H(dp\ot\alpha-p\ot d\alpha)\\
&=dhp\ot g\alpha+hp\ot dg\alpha-dfp\ot H\alpha+fp\ot dH\alpha\\&\quad+dp\ot H\alpha-hp\ot gd\alpha+fp\ot Hd\alpha\\
&=(gp-fp)\ot g\alpha+fp\ot dH\alpha+fp\ot Hd\alpha\\
&=gp\ot g\alpha-fp\ot f\alpha.
\end{align*}
\end{proof}

\begin{proposition}
If $(P,d)$ and $(Q,d_Q)$ are both admissible then $$HC^n(C^\bullet(d,\Fc))\simeq HC^n(C^\bullet(d_Q,\Fc)).$$
\end{proposition}
\begin{proof}
Suppose that $(P,d)$ and $(Q,d_Q)$ are admissible. Then since $P$ is projective and $d_Q$ is surjective, there exists $\alpha:P\rightarrow Q$  such that $d_Q\alpha=d_P$.  Let $\beta$ be any other lifting of $d_P$ to $Q$.  Since $\alpha-\beta$ has image in the kernel of $d_Q$ and by the admissibility of $(Q,d_Q)$ the kernel is equal to the image of $d_0-d_1$, we have a lifting by the projectivity of $P$, namely $h:P\rightarrow Q^{\ot 2}$ such that $\alpha-\beta=dh$.  We are in the setting of Lemma \ref{homotopy}.  Thus the maps $\alpha$ and $\beta$ between $C_n(P)$ and $C_n(Q)$ are chain homotopic.  Thus $$HH^n(\alpha^*)=HH^n(\beta^*)$$ and so $$HC^n(\alpha^*)=HC^n(\beta^*).$$ We've shown that there is always a canonical map $HC^n(C^\bullet(d,\Fc))\leftarrow HC^n(C^\bullet(d_Q,\Fc))$ for any pair of admissible $(P,d)$ and $(Q,d_Q)$.  Since the canonical map from $HC^n(C^\bullet(d,\Fc))$ to itself is $Id$, so we are done.
\end{proof}

This leads to the following definition of cyclic cohomology viewed as a contravariant functor from $\Zc'_\Dc\Dc^{*}$ to vector spaces:
\begin{definition}\label{cyclicofreopcenter}
Suppose that $\Dc$ possesses an admissible pair, we say that $\Dc$ is admissible. Let $\Fc\in \Zc'_\Dc\Dc^{*}$. Define $$HC_\Dc^n(\Fc)=HC^n(C^\bullet(d,\Fc))$$ for any admissible $(P,d)$.
\end{definition}

Note that elements of $\Zc'_\Dc\Dc^{*}$ pull back.  More precisely, if $F:\Cc\rightarrow\Dc$ is a monoidal functor and $\Fc\in \Zc'_\Dc\Dc^{*}$ then $$F^*\Fc(-)=\Fc(F(-))$$ is an element of $\Zc'_\Cc\Cc^{*}$.  This also follows more abstractly from the considerations set out below.

If $\Mc$ is a $\Dc$-bimodule, then $F^*\Mc$ is a $\Cc$-bimodule and if $\Fc:\Mc\rightarrow\Mc'$ is a $\Dc$-bimodule map, then so is $F^*\Fc:F^*\Mc\rightarrow F^*\Mc'$.  We have natural $\Cc$-bimodule maps $F^*\Dc^*\rightarrow\Cc^*$ and $\Cc\rightarrow F^*\Dc$.  An $\Fc\in\Zc_\Dc\Dc^{*}$ is the same as a $\Dc$-bimodule map $\Fc:\Dc\rightarrow\Dc^*$ and so we have $\Cc$-bimodule maps $$\Cc\rightarrow F^*\Dc\rightarrow F^*\Dc^*\rightarrow\Cc^*$$ the composition of which is what we called $F^*\Fc$.  While it is possible to define the restricted center condition abstractly as well, it is not worth it in practice.

\begin{lemma}\label{adj1}
Suppose that $\Cc$ is admissible and $F:\Cc\rightarrow\Dc$ is a monoidal functor that has an exact right adjoint ${^*F}$, then $\Dc$ is admissible and if $\Fc\in \Zc'_\Dc\Dc^{*}$ then $$HC_\Dc^n(\Fc)\simeq HC_\Cc^n(F^*\Fc).$$
\end{lemma}
\begin{proof}
If $(P,d)$ is an admissible pair then since $F$ is right exact, so if $$P^{\ot 2}\rightarrow P\rightarrow \id\rightarrow 0$$ is exact then so is  $$F(P)^{\ot 2}\rightarrow F(P)\rightarrow \id\rightarrow 0.$$  If $P$ is projective then $\Hom_\Dc(F(P),-)\simeq\Hom_\Cc(P,{^*F}(-))$ is exact since both ${^*F}(-)$ and $\Hom_\Cc(P,-)$ are, and so $F(P)$ is projective. Thus $(F(P),F(d))$ is also an admissible pair. Furthermore, the rest follows immediately, since $C^\bullet(d,F^*\Fc)\simeq C^\bullet(F(d),\Fc)$ essentially by definition.
\end{proof}

\section{The world of $2$-categories}
Though our intention is to talk about algebras in monoidal categories and $2$-contratraces, it pays to at this point remember that our monoidal category $\Cc$ is a $2$-category with one object.

\subsection{The $2$-category $\bimod_\Cc$}
Less trivially, we will actually be concerned with a particular $2$-category associated to $\Cc$, namely $\bimod_\Cc$.  Before we continue we should assume that $\ot$ in $\Cc$ is right exact in both variables.  The objects in $\bimod_\Cc$ are unital associative algebras in $\Cc$ and the $1$-morphisms from $B$ to $A$ are $\Hm(B,A)=\bimod_\Cc(A,B)$, the category of left $A$ and right $B$ bimodules in $\Cc$.  Since the $\ot$ is right exact so $\bimod_\Cc(A,B)$ is abelian and the forgetful functor to $\Cc$ is both exact and reflects exactness. The composition $$\Hm(B,C)\times \Hm(A,B)\rightarrow \Hm(A,C)$$ is given by $$\bimod_\Cc(C,B)\times\bimod_\Cc(B,A)\rightarrow \bimod_\Cc(C,A)$$ $$(S,T)\mapsto S\ot_B T$$ where $-\ot_B-$ is the cokernel of the difference map $-\ot B\ot -\rightarrow -\ot -$ and is thus also right exact.

\subsection{More generally}
Keeping the above example in mind, suppose that $\Bc$ is a $2$-category.  Let ${_y\Bc_x}=\Hm(x,y)$ and $\Bc_x=\Hm(x,x)$ where $\Bc_x$ is a monoidal category of the type we are familiar with. The notion of a $\Bc$-bimodule is readily generalized from that of bimodule categories over monoidal categories and would yield for every $x\in\Bc$ a $\Bc_x$-bimodule.  Roughly speaking we have categories ${_y\Mc_x}$ and actions: $${_t\Bc_y}\times{_y\Mc_x}\times{_x\Bc_z}\rightarrow{_t\Mc_z}.$$  Obviously $\Bc$ itself is a bimodule over itself.

Define a different $\Bc$-bimodule: $\Bc^*$ by $${_y\Bc^*_x}=Fun({_x\Bc^{op}_y},\vect)^{op}$$  and $${_t\Bc_y}\times{_y\Bc^*_x}\times{_x\Bc_z}\rightarrow{_t\Bc^*_z}$$ $$(p,F(-),q)\mapsto F(q\circ -\circ p).$$  Note that ${_x\Bc^*_x}=\Bc_x^*$ as expected.

Observe that the center of $\Bc^*$ consists, in particular, of a collection $\Fc_x\in\Zc_{\Bc_x}\Bc_x^*$ together with isomorphisms $$\tau_{q,p}:\Fc_x(q\circ p)\rightarrow\Fc_y(p\circ q)$$ for $p\in{_y\Bc_x}$ and $q\in{_x\Bc_y}$.  We say that $\Fc_\bullet\in\Zc'_{\Bc}\Bc^*$, i.e., in the restricted center of $\Bc^*$ if as expected $$\tau_{p,q}\circ\tau_{q,p}=Id_{\Fc_x(q\circ p)}$$ for all $x,y$ and $p,q$.  In particular, if $\Fc_\bullet\in\Zc'_{\Bc}\Bc^*$ then $\Fc_x\in\Zc'_{\Bc_x}\Bc_x^*$.

\begin{definition}
We say that an object  $x$ in $\Bc$ is admissible if $\Bc_x$ is admissible.

Assuming that $x$ is admissible define the cyclic cohomology of $x$ with coefficients in $\Fc_\bullet\in\Zc'_{\Bc}\Bc^*$ as follows: $$HC^n(x,\Fc_\bullet)=HC^n_{\Bc_x}(\Fc_x).$$
\end{definition}

We are ready for a precursor to Morita invariance.

\begin{proposition}\label{morita1}
Let $x,y$ be objects in $\Bc$, suppose that there exist $p\in{_y\Bc_x}$ and $q\in{_x\Bc_y}$ such that we have isomorphisms $\id_x\simeq q\circ p$ and $p\circ q\simeq\id_y$ then $x$ is admissible if and only if $y$ is admissible and in either case $$HC^n(x,\Fc)\simeq HC^n(y,\Fc).$$

\end{proposition}

\begin{proof}
Let $F:\Bc_x\rightarrow \Bc_y$ be given by $p\circ -\circ q$, then it is a monoidal equivalence.  Thus $x$ is admissible if and only if $y$ is admissible.  Clearly ${^*F}=q\circ -\circ p$ and is exact.  Note that $F^*\Fc_y=\Fc_y(p\circ -\circ q)\simeq\Fc_x(-\circ q\circ p)\simeq\Fc_x(-)$ is an equivalence in $\Zc'_{\Bc_x}\Bc_x^*$ and so we are done by Lemma \ref{adj1}.
\end{proof}

Note that, just as before, if $F:\Bc\rightarrow\Bc'$ is a $2$-functor then $F^*\Fc_\bullet\in \Zc'_{\Bc}\Bc^*$ if $\Fc_\bullet\in \Zc'_{\Bc'}\Bc'^*$.

\begin{definition}
We say that a $2$-functor $F:\Bc\rightarrow\Bc'$ is split if for all $x\in\Bc$ the functor $F_x:\Bc_x\rightarrow\Bc'_{F(x)}$ has an exact right adjoint.
\end{definition}

And now an adjunction, an analogue of Lemma \ref{adj1}, that leads to $2$-Morita invariance.

\begin{proposition}\label{adj2}
Let $F:\Bc\rightarrow\Bc'$ be split, then if $x$ is admissible, we have $$HC^n(x, F^*\Fc_\bullet)\simeq HC^n(F(x), \Fc_\bullet).$$
\end{proposition}

\begin{proof}
This is Lemma \ref{adj1} applied to $F_x:\Bc_x\rightarrow\Bc'_{F(x)}$.
\end{proof}

\subsection{The symmetric $2$-contratraces on the $2$-category $\bimod_\Cc$}
We keep assuming that the product in $\Cc$ is right exact.  Let $\Fc\in\Zc'_\Cc\Cc^*$ and assume that $\Fc\in Fun(\Cc^{op},\vect)$ is left exact.  Let $A$ be a unital associative algebra in $\Cc$.  Define a new functor $\Fc_A\in Fun((\bimod_\Cc A)^{op},\vect)$ as the equalizer of the two actions of $A$, i.e., $\tau_{A,S}\circ\Fc(l_a)$ and $\Fc(r_a)$ from $\Fc(S)$ to $\Fc(S\ot A)$.  Using left exactness of $\Fc$ in addition to the fact that  $\Fc\in\Zc'_\Cc\Cc^*$ we get that $\Fc_A\in\Zc'_{\bimod_\Cc A}(\bimod_\Cc A^*)$ and $\Fc_A$ is still left exact. In fact we also have $$\Fc_A(S\ot_B T)\simeq\Fc_B(T\ot_A S)$$ thus this procedure defines an element $\Fc_\bullet\in\Zc'_{\bimod_\Cc}(\bimod_\Cc^*)$.  It is not hard to see that  any left exact $\Fc_\bullet\in\Zc'_{\bimod_\Cc}(\bimod_\Cc^*)$ arises in this way.

More formally, the above can be obtained from the observation that under suitable exactness assumptions $\Cc$-bimodules and maps between them determine and are in turn determined by the corresponding $\bimod_{-}$ variants.  Roughly speaking, for a $\Cc$-bimodule $\Mc$ with right exact actions, we have a $\bimod_\Cc$-bimodule $\bimod_\Mc$ and a right exact $\Cc$-bimodule functor $F:\Mc\rightarrow\Mc'$ induces a $\bimod_\Cc$-bimodule map $\bimod_F:\bimod_\Mc\rightarrow\bimod_{\Mc'}$.

\subsection{The old and the new}
The purpose of this section is to connect the notions above to the definitions of \cite{hks}.  Let $\Bc$ be a $2$-category and $\Fc_\bullet\in\Zc'_{\Bc}(\Bc^*)$.  Let $x,y\in\Bc$ and $p\in{_y\Bc_x}, q\in{_x\Bc_y}$ such that we have adjunctions $\iota:\id_x\rightarrow q\circ p$ and $\epsilon: p\circ q\rightarrow\id_y$ with the usual conditions: the composition $$q\rightarrow (q\circ p)\circ q\simeq q\circ(p\circ q)\rightarrow q$$ is identity, and similarly for $p$.  Under these assumptions $A=q\circ p$ is a unital associative algebra in $\Bc_x$ and $C=p\circ q$ is a counital coassociative coalgebra in $\Bc_y$.  Furthermore, we have an isomorphism of precocyclic objects  $$\tau_{q,p\circ A^{\circ n}}:\Fc_x(A^{\circ n+1})\rightarrow C^n(\epsilon,\Fc_y)$$ where the structure on the left is from \cite{hks} as explained in Section \ref{review} and on the right from Section \ref{precocyclic}.  In fact the isomorphism is of cocyclic objects, with the left already such, and the right obtaining the codegeneracies from the comultiplication on $C$.

\begin{remark}\rm{
The isomorphism $\tau_{q,p\circ A^{\circ n}}$ being a ``flip" itself explains the presence of the ``flipover coface" on the left and its absence on the right.
}\end{remark}

The following definition is not equivalent to the one in \cite{hks} unless additional assumptions are made.
\begin{definition}
Let $A$ be a unital associative algebra in $\Cc$ and $\Fc\in\Zc'_\Cc\Cc^*$ left exact. Then we obtain an $\Fc_\bullet\in\Zc'_{\bimod_\Cc}(\bimod_\Cc^*)$ and define the cyclic cohomology of $A$ with coefficients in $\Fc$ to be $$HC^n(A,\Fc):=HC^n(A,\Fc_\bullet)$$ provided that $A$ is admissible as an object of $\bimod_\Cc$. We often write $HC^n_\Cc(A,\Fc)$ to emphasize the monoidal category.
\end{definition}

\begin{proposition}\label{oldvsnew}
Let $\id\in\Cc$ be projective, and $\Fc$ a left exact symmetric $2$-contratrace then every algebra is admissible and $$HC^n_{old}(A,\Fc)\simeq HC^n(A,\Fc).$$
\end{proposition}
\begin{proof}
Recall the discussion in the beginning of this section.  Let $\Bc=\bimod_\Cc$, $x=\id$, $y=A$, $p=A^l$ (where $A$ is considered as a left $A$-module), $q=A^r$ (where $A$ is considered as a right $A$-module).  So $A=A$ with $\iota:\id\rightarrow A$ in $\Cc$ and $C=A\ot A$ with $\epsilon=m:A\ot A\rightarrow A$ in $\bimod_\Cc A$. Note that $$C^{\ot 2}\rightarrow C\rightarrow \id\rightarrow 0$$ is exact in $\bimod_\Cc A$ while $C$ is projective in $\bimod_\Cc A$ if  $1$ is projective in $\Cc$.  So $(C,\epsilon)$ is admissible and the proposition follows from the isomorphism $\tau_{q,p\circ A^{\circ n}}$ of precocyclic objects above.
\end{proof}

Again, consider the situation described in the beginning of this section.  In addition, assume that $$C^{\circ 2}\xrightarrow{\epsilon\ot Id-Id\ot \epsilon}C\xrightarrow{\epsilon}\id_y\rightarrow 0$$ is exact.  The latter is equivalent to $\epsilon$ inducing an isomorphism $p\circ_A q\rightarrow \id_y$.  It is easy to see that $$q\circ -\circ p:\Bc_y\rightarrow\bimod_{\Bc_x}A$$ is a monoidal equivalence with the inverse $p\ot_A -\ot_A q$.  Furthermore, provided that $\Fc_x$ is left exact, $\Fc_y$ identifies with $(\Fc_x)_A$ explaining the earlier remark that in $\bimod_\Cc$ all $\Fc_\bullet$ come from $\Fc$ on $\Cc$.  Let us summarize.

\begin{definition}
We say that $(p,q)$ is an adjoint pair if it satisfies both the adjunction conditions in the beginning of this section and the additional exactness property above.
\end{definition}

\begin{lemma}\label{adjlemma}
Let $\Bc$ be a $2$-category and $\Fc_\bullet\in\Zc'_{\Bc}(\Bc^*)$.  If $x,y\in\Bc$ and $(p,q)$ is an adjoint pair, then provided that $\Fc_x$ is left exact and $B\in\Bc_y$ is admissible, we have $$HC^n_{\Bc_y}(B,\Fc_y)\simeq HC^n_{\bimod_{\Bc_x} A}(q\circ B\circ p,(\Fc_x)_A),$$ where as usual $A=q\circ p$.
\end{lemma}

\subsection{Morita invariance}
When two unital associative algebras $A$ and $B$ are Morita equivalent in $\Cc$, i.e., there are bimodules $P$ and $Q$ with $P\ot_A Q\simeq B$ and $Q\ot_B P\simeq A$, then we expect the cyclic cohomology to be the same.  This is an immediate consequence of  Proposition \ref{morita1} applied to $\bimod_\Cc$.

\begin{proposition}\label{cyclicismoritainvar}
Let $\Cc$ be a monoidal category and $\Fc\in\Zc'_\Cc\Cc^*$ left exact.  Suppose that $A$ and $B$ are Morita equivalent algebras in $\Cc$ and $A$ or $B$ is admissible, then $$HC^n(A,\Fc)\simeq HC^n(B,\Fc).$$
\end{proposition}

\subsection{A duality adjunction}
Suppose that $A$ is a unital associative algebra in $\Cc$, let $\Cc^\vee=\bimod_\Cc A$. Let $$F:\bimod_{\Cc^\vee}\rightarrow\bimod_\Cc$$ be the $2$-functor that sends a unital associative algebra $B$ in $\Cc^\vee$ to $F(B)$ (here $F$ is the forgetful functor that forgets $A$-bimodule structure) in $\Cc$ which is also a unital associative algebra.  Since $\bimod_{\Cc^\vee}(B,B')\simeq\bimod_\Cc(F(B),F(B'))$ the $2$-functor $F$ is not only split, but is a fully faithful embedding.

\begin{proposition}\label{dualityadj}
Let $\Fc$ be a left exact $2$-contratrace on $\Cc$, then if $B$ is an admissible unital associative algebra in $\Cc^\vee$, we have $$HC^n_{\Cc^\vee}(B,\Fc_A)\simeq HC_\Cc^n(F(B),\Fc).$$
\end{proposition}
\begin{proof}
This is an application of Proposition \ref{adj2} to $F:\bimod_{\Cc^\vee}\rightarrow\bimod_\Cc$.
\end{proof}

Another point of view is provided in the following.  It is a composition of Lemma \ref{adjlemma} and Proposition \ref{dualityadj}.

\begin{proposition}\label{dualityadj1}
Let $\Bc$ be a $2$-category and $\Fc_\bullet\in\Zc'_{\Bc}(\Bc^*)$.  If $x,y\in\Bc$ and $(p,q)$ is an adjoint pair, then provided that $\Fc_x$ is left exact and $B\in\Bc_y$ is admissible then $\Fc_y$ is left exact and $q\circ B\circ p\in\Bc_x$ is admissible and we have $$HC^n_{\Bc_y}(B,\Fc_y)\simeq HC^n_{\Bc_x}(q\circ B\circ p,\Fc_x).$$
\end{proposition}

\begin{remark}\rm{
If on the other hand $(p,q)$ is not an adjoint pair, but instead is a Morita pair, i.e, provides a Morita equivalence between $x$ and $y$ in $\Bc$, then we have the following.  The algebra $B\in\Bc_y$ is admissible if and only if $q\circ B\circ p\in\Bc_x$ is admissible, $D\in\Bc_x$ is admissible if and only if $p\circ D\circ q\in\Bc_y$ is admissible, and $\Fc_x$ is left exact if and only if $\Fc_y$ is left exact. Thus we have completely symmetrically, unlike the case of the Proposition \ref{dualityadj1}, $$HC^n_{\Bc_y}(B,\Fc_y)\simeq HC^n_{\Bc_x}(q\circ B\circ p,\Fc_x)$$ and $$HC^n_{\Bc_x}(D,\Fc_x)\simeq HC^n_{\Bc_y}(p\circ D\circ q,\Fc_y).$$
}\end{remark}

\subsection{2-Morita invariance}\label{2moritainvariance}
\begin{definition}\label{2moritadef}
We say that two monoidal categories $\Cc$ and $\Dc$ are $2$-Morita equivalent if $\bimod_\Cc\simeq\bimod_\Dc$.  More precisely, $\Dc\simeq\bimod_\Cc A$ for some unital associative $A\in\Cc$ and there is a unital associative algebra $S\in\bimod_\Cc A$ such that $S$ is Morita equivalent to $\id$ in $\Cc$, i.e., exist bimodules $P$ and $Q$ such that $P\ot Q\simeq S$ and $Q\ot_S P\simeq\id$ in $\Cc$.
\end{definition}

It is immediate from the definition and our prior efforts that we have a bijection of algebras and left exact symmetric $2$-contratraces between those of $\Cc$ and those of $\Dc$ and that this bijection preserves cyclic cohomology.  However we would like to spell out the correspondence more explicitly below.

\begin{proposition}\label{2morita}
Suppose $\Cc$ and $\Dc$ are $2$-Morita equivalent.  Identify $\Dc$ with $\bimod_\Cc A$, and let $(P,Q)$ be the Morita pair from Definition \ref{2moritadef}.  Let $D$ be admissible in $\Cc$ and let $\Fc$ be a left exact symmetric $2$-contratrace, then $$HC^n_\Cc(D,\Fc)\simeq HC^n_{\bimod_\Cc A}(P\ot D\ot Q,\Fc_A).$$
\end{proposition}

\begin{proof}
Let $\Bc$ be a $2$-category. Consider the following diagram in $\Bc$:
$$\xymatrix{
x\ar[dr]_{p''}\ar[rr]^p & & y\ar[dl]^{p'}\\
&z &
}$$ with omitted $q$'s pointing in the opposite direction to the corresponding $p$'s.  Let $(p,q)$ and $(p',q')$ be adjoint pairs and $(p'',q'')$ be a Morita pair.  Let $D$ be an algebra in $\Bc_x$ and let $\Fc_x$ be left exact, then \begin{align*}
HC^n_{\Bc_x}(D,\Fc_x)&\simeq HC^n_{\Bc_z}(p''\circ D\circ q'',\Fc_z)\\
&\simeq  HC^n_{\Bc_y}(q'\circ p''\circ D\circ q''\circ p',\Fc_y)\\
&\simeq  HC^n_{\Bc_x}(q\circ q'\circ p''\circ D\circ q''\circ p'\circ p,\Fc_x).\\
\end{align*}
Now let $\Bc=\bimod_\Cc$, $x=\id$, $y=A$, $z=S$, $(p,q)=(A^l,A^r)$, $(p',q')=(S^l,S^r)$ and $(p'',q'')=(P,Q)$, then the latter cohomology is isomorphic to $HC^n_{\bimod_{\Bc_x}A}(q\circ q'\circ p''\circ D\circ q''\circ p'\circ p,(\Fc_x)_A)$ which in turn is isomorphic to the desired cohomology.

\end{proof}

A wealth of examples of $2$-Morita equivalent monoidal categories can be found among fusion categories \cite{EGNO}.  More precisely any two fusion categories that are dual, are $2$-Morita equivalent.  More generally, if $\Cc$ is rigid we have:

\begin{lemma}
If $\Cc$ is rigid with $\id$ simple and projective, then for any algebra $A$ in $\Cc$, we have that $\Cc$ is $2$-Morita equivalent to $\bimod_\Cc A$.
\end{lemma}
\begin{proof}
Let $A\in\Cc$ be a unital associative algebra, then $A\ot A^*$ is a unital associative algebra in $\bimod_\Cc A$.  We claim that $A\ot A^*$ is Morita equivalent to $\id$ in $\Cc$.  Consider $A\in\bimod_\Cc(A\ot A^*,\id)$ and $A^*\in\bimod_\Cc(\id,A\ot A^*)$, then $A\ot_\id A^*\simeq A\ot A^*$.  Let $P=A^*\ot A$ let $d$ denote the evaluation map to $\id$, it is surjective since $\id$ is simple.  Then $A^*\ot_{A\ot A^*}A\simeq\id$ if $$P^{\ot 2}\rightarrow P\rightarrow\id\rightarrow 0$$ is exact.  Since $\id$ is projective so $d$ splits creating a homotopy as before.
\end{proof}

If $H$ is a finite dimensional Hopf algebra, then its category of left modules is $2$-Morita equivalent to its category of right comodules, or alternatively to the category of left modules over the dual Hopf algebra $H^*$.  If $H$ is not finite dimensional then this is no longer true, but an adjunction does exist and is interesting.

\subsection{A natural example of adjunction}\label{natadj}
We can consider a natural example of the Proposition \ref{dualityadj} above.  Let $\Cc=\Mc^H$ and $A=H$ then by the Fundamental theorem of Hopf modules we have $\Cc^\vee\simeq \hmod$.  Let $B$ be an $H$-module algebra and note that $F(B)$ is $B\rtimes H$ as an $H$-comodule algebra.  Thus if $\Fc(-)=\Hom(-,M)^H$ there exists a left exact $2$-contratrace $\Fc_H$ on $\hmod$ such that $$HC^n(B,\Fc_H)\simeq HC^n(B\rtimes H, \Fc)$$ where $M\in{_H SAYD^H}$ and the right hand side is the $B$-type theory of \cite{hkrs}.  Note that we need the Lemma below that ensures that any $B$ is admissible.

\begin{lemma}\label{admissiblehmod}
Let $\Cc={_\Hc\Mc}$ where $\Hc$ is a weak Hopf algebra with a separable base $R$, then any $A$ unital associative is admissible.
\end{lemma}

\begin{proof}
We need to construct an admissible pair in $\bimod_\Cc A=:\Dc$.  Note that any complex obtained from a surjective $d:P\rightarrow A$ as follows is a resolution. Consider $$(P^{\ot_A n+1},d=\sum_{i=0}^n(-1)^i\delta_i'),$$ where $d: P\rightarrow A$ and $\delta_i'=Id^{\ot_A i}\ot_A d\ot_A Id^{\ot_A n-i}$.  Indeed, let $\widetilde{1}$ be such that $d(\widetilde{1})=1$ and observe that $d(\widetilde{1}\ot_A\alpha)=\alpha-\widetilde{1}\ot_A d\alpha$.  For projectivity we proceed as follows.  Let $Q\rightarrow \id$ be a surjection of a projective $Q$ onto $\id$ in $\Cc$; we can take $Q=\Hc$ with the counit map.

Then $P=A\ot_R Q\ot_R A\rightarrow A$ in $\Dc$ is exactly what we need.  Namely, $P$ is projective in $\Dc$, surjects onto the unit, and $P^{\ot_A i}=A\ot_R(Q\ot_R A)^{\ot_R i}$ is projective in $\Dc$ if $(Q\ot_R A)^{\ot_R i-1}\ot_R Q$ is projective in $\Cc$.  But for any $W$ in ${_\Hc\Mc}$, we have that  $W\ot_R Q$ is projective in ${_\Hc\Mc}$ if $Q$ is.  This follows from the observation that $\Hom_\Hc(W\ot_R Q,-)$ is a composition of the internal Hom, namely $\Hom^r_R(W,-)$ and $\Hom_{\Hc}(Q,-)$ both of which are exact.  See Section \ref{adjactionsection} for a discussion of internal Homs.
\end{proof}

There is no guarantee that $(A\ot A,m)$ is any longer admissible, and thus that the old definition of cyclic cohomology will agree with the invariant one. On the other hand, from the proof above we see that the entire resolution consists of projective objects and so the new definition will have a derived functor interpretation of the Hochschild part of the cohomology.  More precisely, we have:

\begin{lemma}\label{derint}
Consider $\Cc={_\Hc\Mc}$ where $\Hc$ is a weak Hopf algebra with a separable base $R$, let $A$ be a unital associative algebra, and $\Fc$ a left exact symmetric $2$-contratrace, then $\Fc_A$ is a left exact functor $\Fc_A:(\bimod_\Cc A)^{op}\rightarrow\vect$ and $$HH^n(A,\Fc)\simeq R^n\Fc_A(A).$$
\end{lemma}

\begin{remark}\rm{
For the case above, namely when $\Hc$ is a weak Hopf algebra with a separable base, the pair $(A\ot A, m)$ is admissible in $\bimod_\Cc A$ for all $A\in\Cc$ if and only if $\id\in\hmod$ is projective. By the existence of exact internal Homs, the latter happens if and only if everything is projective, which is equivalent to $\hmod$ being semi-simple.  This is not to say that there could not be an $A$ for which $(A\ot A, m)$ is admissible in $\bimod_\Cc A$ even if $\id\in\Cc$ is not projective.
}\end{remark}

But we digress, let us get back to our left exact symmetric $2$-contratrace on $\hmod$, that we called $\Fc_H$ above.  Recall that it is obtained from $\Hom(-,M)^H$ on $\Mc^H$ via adjunction.  We can describe $\Fc_H$ quite explicitly.  If $M\in{_H SAYD^H}$ then $\Hom(-,M)^H$ obtains a symmetric structure via $$\tau f(v\ot w)=S^{-1}(w_1)f(w_0\ot v).$$  Let $V\in\hmod$, then $$\Fc_H(V)=\{f\in \Hom(V\ot H,M)^H| f(v\ot xy)=S^{-1}(y^3)f(y^1 v\ot y^2 x)\}.$$  Let $\widehat{M}=\Hom(H,M)^H$ and define a left $H$-module structure on it by $$y\cdot\varphi(-)=y^2\varphi(S(y^3)-y^1)$$ so that we have $$\Fc_H(V)=\Hom_H(V,\widehat{M}).$$  The symmetric structure is as follows: $$\tau\varphi(w\ot v)(h)=\varphi(S^{-1}(h^1)v\ot w)(h^2).$$

Compare this with the ``usual" way of getting a symmetric left exact $2$-contratrace from an SAYD $M$, namely $$\Fc'(V)=\Hom_H(M\ot V,\id)\simeq\Hom_H(V,{^*M})$$ that yields the type $A$ theory on $\hmod$.  Note that the two ways are not the same, they are in some sense dual to each other.  On the other hand, if we consider ${_{ad}H^\triangle}\in{_H SAYD^H}$ with the action $hx=h^2 xS(h^1)$ and the coaction given by the comultiplication $\triangle$, then a straightforward computation shows that as $2$-contratraces we have $$\Hom_H({_{ad}H^\triangle}\ot -,\id)\simeq\Hom_H(-,\widehat{{_{ad}H^\triangle}}),$$ so that:

\begin{lemma}\label{duallemma1}
Let $B$ be an $H$-module algebra then $$HC^n_H(B,{_{ad}H^\triangle})\simeq HC^{n,H}(B\rtimes H,{_{ad}H^\triangle}).$$
\end{lemma}

We will come back to this example later, now let us examine the particular situation of $H=kG$, the group algebra of a discrete group.  We have that $\hmod=Rep(G)$ and $\Mc^H=\vect_G$.  A bit of care has to be taken with the latter as $G$ might be infinite so that the correspondence between $H$-comodules and sheaves on $G$ is via $\Gamma_c(G,-)$ and the convolution product is defined by $V\star W=m_!(V\boxtimes W)$ where $m:G\times G\rightarrow G$ is the group operation.

A SAYD module $\Mc$ here is a sheaf on $G$ that is $G$-equivariant with respect to the adjoint action with the additional assumption that $g\in G$ acts trivially on $\Mc_g$, the fiber of $\Mc$ at $g$.  Note that while as an $H$-module and comodule $M$ is given by $\Gamma_c(G,\Mc)$, we have that $$\widehat{M}\simeq\Gamma(G,\Mc)$$ which is obviously an $H$-module, though not a comodule.  More explicitly, we have $$\Fc_G(V)=\Hom_G(V,\widehat{M})=\{(\varphi_x:V\rightarrow \Mc_x)_{x\in G}|g\varphi_x(g^{-1}v)=\varphi_{gxg^{-1}}(v)\}$$ and the symmetric structure simplifies to $$(\tau\varphi)_x(w\ot v)=\varphi_x(x^{-1}v\ot w).$$

So that if $V\in Rep(G)$ and we consider it and its dual as SAYD's concentrated at $e\in G$ then $$HC^n_{Rep(G)}(A,\Hom_G(V_e\ot -,\id))\simeq HC^n_{\vect_G}(A\rtimes G,\Hom_{\vect_G}(-,V_e^*))$$ where the left hand side is the type $A$ theory and the right hand side is the type $B$ theory. Thus we have:

\begin{lemma}\label{duallemma2}
Let $A$ be a $G$-equivariant algebra then $$HC^n_G(A,V_e)\simeq HC^{n,G}(A\rtimes G, V_e^*).$$
\end{lemma}

We will come back to this in Section \ref{derhamsection}.

More generally, let $\Mc^\vee=((-)^{-1*}\Mc)^*$, i.e., $\Mc^\vee_g=\Mc^*_{g^{-1}}$.  We see that $\Mc^\vee$ is an SAYD as well and one has an immediate generalization of Lemma \ref{duallemma2}.

\begin{lemma}\label{duallemma3}
Let $A$ be a $G$-equivariant algebra then $$HC^n_G(A,\Gamma_c(G,\Mc))\simeq HC^{n,G}(A\rtimes G, \Gamma_c(G,\Mc^\vee)).$$
\end{lemma}

The case of the modular pair in involution (MPI) follows easily from the above.  Let $\chi$ be a character of $G$ and $x\in Z(G)$ such that $\chi(x)=1$ then both $k_{\chi,x}$ and $k_{-\chi,x^{-1}}$ are MPIs and $$HC^n_G(A,k_{\chi,x})\simeq HC^{n,G}(A\rtimes G, k_{-\chi,x^{-1}}).$$

\begin{remark}\label{diff}\rm{
If we apply the above considerations to the case of a discrete but infinite $G$ (we can keep the example $G=\mathbb{Z}$ in mind), then we obtain an indication that the definition we use in this paper is not only different, but has its uses.  More precisely, we observe that $HC^\bullet_{G, old}(k,k_e)\simeq HC^\bullet(k)$ while since $\vect_G$ is semi-simple: $$HC^\bullet_G(k,k_e)\simeq HC^{\bullet,G}(kG,k_e)\simeq HC_{old}^{\bullet,G}(kG,k_e)\simeq HC^\bullet(kBG)$$ and the latter is isomorphic to $H^\bullet(G,k)\ot HC^\bullet(k)$ by the Karoubi's theorem \cite{weibel}.  Note that $H^\bullet(\mathbb{Z},k)\simeq Ext^\bullet_{k[x,x^{-1}]}(k_1,k_1)\simeq k\oplus k[-1]$ thus showing that $HC^\bullet_{\mathbb{Z}, old}(k,k_e)$ is different from $HC^\bullet_\mathbb{Z}(k,k_e)$.
}\end{remark}

\subsection{Cyclic cohomology and monoidal functors}
We have seen that given an algebra $A\in\Cc$ and considering  $\bimod_\Cc A$ and $\Cc$, algebras move to the right and contratraces to the left creating an adjunction that preserves cyclic cohomology.  We now examine what happens in the presence of a monoidal functor between two categories.

\begin{lemma}\label{adj}
Let $\Cc$ and $\Dc$ be monoidal categories and $F:\Cc\rightarrow \Dc$ be a monoidal functor with an exact right adjoint $\dF$.  Then we have $$F:\bimod_\Cc\rightarrow\bimod_\Dc$$ a split $2$-functor induced by the original $F$.
\end{lemma}

\begin{proof}
Recall that we always assume that the $\ot$'s are right exact, so that $\bimod_\Cc$ and $\bimod_\Dc$ are defined.  Since $F$ is monoidal it sends algebras to algebras and since it is right exact, so it preserves tensor products over algebras and so defines a $2$-functor. Now consider the monoidal functor $F:\bimod_\Cc A\rightarrow\bimod_\Dc F(A)$, that simply sends $T$ to $F(T)$.  We claim that it has an exact right adjoint that is just $\dF$, which would automatically be exact.  We need to define $A$-bimodule structure on $\dF(S)$ which is just the adjoint of the map: $$F(A\ot\dF(S)\ot A)\simeq F(A)\ot F\dF(S)\ot F(A)\rightarrow F(A)\ot S\ot F(A)\rightarrow S.$$ It now suffices to check that the maps $F\dF(S)\rightarrow S$ and $T\rightarrow\dF F(T)$ are bimodule maps.

Consider the diagram:
$$\xymatrix{
F(A)\ot F\dF(S)\ot F(A)\ar[r]\ar[d]&F(A)\ot S\ot F(A)\ar[d]\\
F\dF(S)\ar[r]&S\\
}$$ that commutes if the following diagram commutes:
$$\xymatrix{
A\ot  \dF(S)\ot A\ar[r]\ar[d]&\dF(F(A)\ot S\ot F(A))\ar[d]\\
\dF(S)\ar[r]&\dF(S)\\
}$$ which commutes by the definition of the $F(A)$-bimodule structure.

Now consider the diagram below:
$$\xymatrix{
A\ot T\ot A\ar[r]\ar[dd]& A\ot \dF F(T)\ot A\ar[rd]\ar[dd]&\\
& & \dF(F(A)\ot F(T)\ot F(A))\ar[ld]\\
T\ar[r]&\dF F(T) &\\
}$$ where the triangle commutes by definition and the square commutes if the following commutes:
$$\xymatrix{
A\ot T\ot A\ar[r]\ar[d]&A\ot \dF F(T)\ot A\ar[r]&\dF F(A\ot T\ot A)\ar[d]\\
T\ar[rr]& & \dF F(T)\\
}$$ which commutes by naturality.

\end{proof}

\begin{proposition}\label{adjunctionprop1}
Let $F:\Cc\rightarrow\Dc$ be a monoidal functor with an exact right adjoint, $\Fc\in\Zc'_\Dc(\Dc^*)$ left exact and $A\in\Cc$ an admissible unital associative algebra, then $F^*\Fc$ is left exact and $$HC^n_\Cc(A,F^*\Fc)\simeq HC^n_\Dc(F(A), \Fc).$$
\end{proposition}
\begin{proof}
By Lemma \ref{adj} we have a split $2$-functor $F:\bimod_\Cc\rightarrow\bimod_\Dc$.  Note that $F^*\Fc=\Fc\circ F$ is left exact since $F$ is right exact (despite looking wrong, this is not a typo).  Observe that $(F^*\Fc)_A\simeq F^*(\Fc_{F(A)})$ and so the result follows from Proposition \ref{adj2}.
\end{proof}

Thus given a suitable $F:\Cc\rightarrow \Dc$ we again have that algebras move to the right and contratraces to the left creating an adjunction that preserves cyclic cohomology.

\subsection{Some examples}
This presentation provides an alternative point of view to \cite{js}.  Furthermore, see Remark \ref{twist} below for a case, in fact one of the main motivators for this paper, that does not fit into the Hopf algebra framework.

\begin{example}\label{ex1}\rm{
Let $F:\hmod\rightarrow\vect$ be the fiber functor, and note that $\Hom_k(H,-)$ is its exact right adjoint.  Let $\Fc=(-)^*$ be a symmetric $2$-contratrace on $\vect$, which is exact.  It is straightforward to show that $$F^*\Fc\simeq \Hom_H({_m H^{ad}}\ot -,\id),$$ where ${_m H^{ad}}\in{_H SAYD^H}$ is such that the action is via multiplication and the coaction is given by $\rho(h)=h^2\ot h^3S(h^1)$. Thus we have that for an $H$-module algebra $A$, we get $$HC^n_H(A, {_m H^{ad}})\simeq HC^n(A),$$ where the latter is the usual cyclic cohomology of the algebra $A$ with $H$-module structure disregarded.}
\end{example}

More generally, if $K\to H$ is a Hopf algebra morphism then we have an $F:\hmod\to{_K\Mc}$ such that given an $N\in{_K SAYD^K}$ and $\Fc(-)=\Hom_K(N\ot-,\id)$ we see that $$F^*\Fc(-)\simeq\Hom_H(\text{Ind}_K^H N\ot-,\id)$$ where the induction functor for SAYDs is defined in \cite{js}.  The Example \ref{ex1} is obtained from the unit $k\to H$.

\begin{example}\label{fiberexample}\rm{
Let $F:\Mc^H\rightarrow\vect$ be the fiber functor, and note that $-\ot H$ is its exact right adjoint.  Let $\Fc=(-)^*$ be a symmetric $2$-contratrace on $\vect$, which is just as exact as before.  Now we have that $$F^*\Fc\simeq \Hom(-,{_{ad}H^\triangle})^H.$$  Thus we see that for an $H$-comodule algebra $B$, we get $$HC^{n,H}(B, {_{ad}H^\triangle})\simeq HC^n(B).$$

Using Lemma \ref{duallemma1} we can also say that $$HC^n_H(A,{_{ad}H^\triangle})\simeq HC^n(A\rtimes H).$$}
\end{example}

Again, if $K\to H$ is a Hopf algebra morphism then we have an $F:\Mc^K\to\Mc^H$ such that given an $M\in{_H SAYD^H}$ and $\Fc(-)=\Hom(-,M)^H$ we see that $$F^*\Fc(-)\simeq\Hom(-, \text{Res}_K^H M)^K$$ where the restriction functor for SAYDs is defined in \cite{js}.  The first part of the Example \ref{fiberexample} is obtained from the counit $H\to k$.

\begin{remark}\label{twist}\rm{
Let $G$ be a finite group.  We can apply the cyclic cohomology considerations of \cite{hks} and the present paper to the monoidal category $\vect_G^\omega$, see \cite{EGNO} for the details of the definition of the latter.  Roughly speaking, $\vect_G^\omega$ is a twisted version of $\vect_G$ in the sense that the associator has been modified by the $\omega\in H^3(G,k^\times)$.  In fact one thinks of $\omega$ as a $3$-cocycle, with the understanding that   $\vect_G^\omega$ only depends, up to equivalence, on the cohomology class of $\omega$.  In this case, the AYDs are well understood (they coincide with the elements of the center) and they are twisted representations of the inertia groupoid $IG$ of the groupoid $pt/G$  associated to the group $G$.  More precisely, the transgression of $\omega$ is an element $t\omega\in H^2(IG,k^\times)$ and thus defines the required twisting.  The stability condition is the same as in the untwisted case.  Let us denote the SAYDs by $Rep^{t\omega}_0 IG$.

To apply the formalism we need algebras, and one way is to obtain them from a $\pi:G'\to G$ a group morphism such that $\pi^*\omega=0$.  We assume that $G'$ is also finite. Namely, $\pi$ induces a monoidal functor $\pi_*:\vect_{G'}\to\vect_G^\omega$ with an exact right adjoint $\pi^*$.  We also have the functor $\pi^*: Rep_0^{t\omega}IG\to Rep_0 IG'$.  Given an algebra $A\in\vect_{G'}$ such as $kG'$, we obtain an algebra $\pi_* A\in\vect_G^\omega$.  Thus taking an $M\in Rep^{t\omega}_0 IG$ we obtain $$HC^n_{\vect_G^\omega}(\pi_* A, \Hom_{\vect_G^\omega}(-,M))\simeq HC^{n,G'}(A,\pi^* M).$$  Note that if $\omega=0$ then the functor $\pi^*:Rep_0 IG\to Rep_0 IG'$ coincides with $\text{Res}^{kG}_{kG'}$ of \cite{js}.

}\end{remark}

\subsection{A more unified perspective}\label{adjunctionsection}
Let $\Cc$ and $\Dc$ be monoidal categories, and suppose that we are given a split $2$-functor $F:\bimod_\Cc\rightarrow\bimod_\Dc$.  Let $\Fc_\bullet$ be a symmetric $2$-contratrace on $\bimod_\Dc$ such that for all $B\in\bimod_\Dc$, we have that $\Fc_B$ is left exact.  It is not hard to show that any such data is equivalent to a monoidal functor $F:\Cc\rightarrow\bimod_\Dc A$, for some $A\in\Dc$ with an existing exact $\dF$ and an $\Fc\in\Zc'_\Dc(\Dc^*)$ that is left exact.  We mention this only in passing, as motivation for restricting ourselves to the following.

\begin{definition}
We say that a $(\Cc,\Dc)$-bimodule category $\Nc$ is admissible if it is equivalent to $A_\Dc\md$ as a right $\Dc$-module category and the resulting monoidal functor $\Cc\rightarrow\bimod_\Dc A$ admits an exact right adjoint.
\end{definition}

Note that the above notion is not symmetric in general in $\Cc$ and $\Dc$.  We should instead think of $\Nc$ as a $1$-morphism from $\Cc$ to $\Dc$ in the $3$-category of monoidal categories thus defined, a sort of higher analogue of $\bimod_\vect$.

\begin{remark}\rm{
If $\Nc$ is a $(\Cc,\Dc)$-bimodule category and $\Mc$ is a $(\Dc,\Ec)$-bimodule category then the existence of $\Nc\boxtimes_\Dc\Mc$ is a subtle question and the object is not in general defined. However if both are admissible, then we see that by interpreting the bimodules $\Nc$ and $\Mc$ as maps from $\bimod_\Cc$ to $\bimod_\Dc$, and from $\bimod_\Dc$ to $\bimod_\Ec$ respectively, that the composition is defined and is a map from $\bimod_\Cc$ to $\bimod_\Ec$.  Unraveling, we have an explicit construction of an admissible $(\Cc,\Ec)$-bimodule category $\Nc\boxtimes_\Dc\Mc$.  More precisely, let $\Nc=A_\Dc\md$, $\Mc=A'_\Ec\md$, and $F:\Dc\rightarrow\bimod_\Ec A'$.  Consider $F(A)$ as an algebra in $\Ec$ and let $$\Nc\boxtimes_\Dc\Mc=F(A)_\Ec\md$$ with $\Cc\rightarrow\bimod_\Ec F(A)$ given by $$\Cc\rightarrow\bimod_\Dc A\rightarrow \bimod_{\bimod_\Ec A'}F(A)\simeq\bimod_\Ec F(A)$$ where the second functor exists and has an exact right adjoint by the proof of Lemma \ref{adj}.
}\end{remark}

Combining Proposition \ref{adjunctionprop1} with Proposition \ref{dualityadj} we observe that $\Nc$ induces $\Nc^*(-)$ on left exact $\Fc\in\Zc'_\Dc(\Dc^*)$ and $\Nc_*(-)$ on algebras in $\Cc$ such that we have the following:

\begin{theorem}\label{sl}
Let $\Nc$ be an admissible $(\Cc,\Dc)$-bimodule, $\Fc\in\Zc'_\Dc(\Dc^*)$ left exact, and $A$ an admissible algebra in $\Cc$, then $$HC_\Cc^n(A,\Nc^*\Fc)\simeq HC^n_\Dc(\Nc_* A,\Fc).$$
\end{theorem}

\begin{example}\label{compositexample}\rm{
Consider the $(\hmod,\vect)$-bimodule $\hmod$ with the obvious actions.  Observe that it is admissible since $\hmod=H_\vect\md$ and the resulting functor $\hmod\rightarrow\bimod_\vect H$ is $V\mapsto V\ot H$ with $x(v\ot h)y=x^1 v\ot x^2 hy$.  It has an exact right adjoint $W\mapsto W_{ad}$, where the left action of $H$ on $W_{ad}$ is as $x\cdot w=x^1 w S(x^2)$.  We see that $$\Nc^*(-)^*\simeq\Hom_H({_{ad}H^\triangle}\ot -,\id)$$ and $$\Nc_* A\simeq A\rtimes H$$ yielding $$HC^n_H(A,{_{ad}H^\triangle})\simeq HC^n(A\rtimes H)$$ again, but this time in one shot.

}\end{example}

The similarity of the above to the considerations of Lemma \ref{duallemma1} and Example \ref{fiberexample} is explained by the observation that the former comes from the admissible $(\hmod,\Mc^H)$-bimodule $\vect$ and the latter from the admissible  $(\Mc^H,\vect)$-bimodule $\vect$.  Now the $(\hmod,\vect)$-bimodule $\hmod$ of the Example \ref{compositexample} is the composition of the two, i.e., we have that as an  $(\hmod,\vect)$-bimodule $$\hmod\simeq \vect\boxtimes_{\Mc^H}\vect.$$

On the other hand the above can be viewed as a special case of the following.  Suppose that $\Cc$ is equipped with an $R$-fiber functor $F$ which generalizes the usual notion of fiber functor by replacing the target $\vect$ with $\bimod_\vect R$.  The usual example of this is the case of weak Hopf algebras (see Section \ref{adjactionsection}).  Now suppose that $F$ admits an exact right adjoint ${^*F}$, so that $R_\vect\md$ is an admissible $(\Cc,\vect)$-bimodule.  Then by Theorem \ref{sl} we get that for an admissible algebra $A\in\Cc$ there is an isomorphism $$HC_\Cc^n(A,\Hom_\Cc(-,{^*F}(R^*)))\simeq HC^n(F(A))$$ since $(-)^*_R\simeq\Hom_{\bimod_\vect R}(-,R^*)$.  Observe that the latter is the standard symmetric $2$-contratrace on $\bimod_\vect R$, i.e., $HH_0(R,-)^*$.

\section{Appendix}
Here we collect some additional considerations that while not directly relevant to the main exposition possess some motivational or explanatory value.

\subsection{Coalgebras}\label{coalgebrasection}

We briefly mention some considerations involved in the coalgebra case.  The first difference is that covariant functors from $\Cc$ to $\vect$ need to be considered in order to obtain a cohomology theory.  Thus we would talk about symmetric $2$-traces.  The requirement for left exactness remains.  The category of $A$-bimodules is replaced with the category of $C$-bicomodules.  Examples of symmetric $2$-traces can be obtained from SAYD modules as before.  Namely, given $N\in{_H SAYD^H}$ we have that $$\Hom_H(N,-)$$ and $$\Hom(\id,N\ot -)^H$$ are left exact symmetric $2$-traces.  Note that neither appears in the usual three Hopf cyclic theory types $A$, $B$, $C$, but something like the former appeared in \cite{ser}.  The ``projective" constructions of Section \ref{precocyclic} should be replaced with the ``injective" analogues.

The purpose of this section  is not to spell out the differences between algebras and coalgebras but rather to make the following observation. Let $A$ be an algebra in $\Cc$ such that $$0\rightarrow \id\rightarrow A\rightarrow A^{\ot 2}$$ is exact, then if $C$ denotes the coalgebra $A\ot A$ in $\Cc^\vee=\bimod_\Cc A$, we have that $\Cc$ is equivalent to the category of $C$ bicomodules in $\Cc^\vee$.  So that in addition to algebras going from $\Cc^\vee$ to $\Cc$ and left exact $2$-contratraces going the other way while preserving cyclic cohomology, we have coalgebras going from  $\Cc$ to $\Cc^\vee$ and left exact $2$-traces going the other way, again preserving cyclic  cohomology.  An example of this is the relationship between $\hmod$ and $\Mc^H$ for a general Hopf algebra $H$.  Recall that if $H$ is finite dimensional then the two are $2$-Morita equivalent, but in general this is the best that we can hope for.

\subsection{Cyclic and de Rham cohomologies with coefficients}\label{derhamsection}

Recall that there exist a series of results \cite{loday} relating cyclic homology of commutative algebras to the de Rham cohomology of the spaces on which these algebras are the algebras of functions.  There are versions both in the smooth and the algebraic settings.  Let us choose the latter for the sake of illustrating some of our motivations.  More precisely, let $A$ be a unital commutative algebra over $k$ of characteristic $0$ such that $V=\text{Spec} A$ is smooth, then we have a functorial identification $$HC_n(A)\simeq\Omega^nA/d\Omega^{n-1}A\oplus\bigoplus_{i\geq 1}H^{n-2i}_{dR}(V).$$  Observe that $HC_n(A)$ is computed from a cyclic object $C_n=A^{\ot n+1}$ \cite{loday}, via the $\lambda$-coinvariants of the Hochschild homological complex associated to $A$.  For our purposes it is more natural to consider the dual situation, i.e., we want to consider $$C^n=\Hom(A^{\ot n+1}, k)$$ which is then a cocyclic object whose cyclic cohomology is obtained via the $\lambda^*$-invariants of the Hochschild cohomological complex associated to $A$.  It is immediate that $$HC^n(A)\simeq HC_n(A)^*$$ and so we get (for $V=\text{Spec}A$) $$HC^n(A)\simeq\text{ker}d^*|_{(\Omega^nA)^*}\oplus\bigoplus_{i\geq 1}H_{n-2i}^{dR}(V),$$ where $H_{n-2i}^{dR}(V)=H^{n-2i}_{dR}(V)^*$.

Let us consider a more general situation, namely we introduce coefficients on both sides.  Let $X$ be a smooth affine variety and let $E$ be a vector bundle on $X$ equipped with a flat connection.  Furthermore, assume that there exists a smooth $G$-principal bundle $\pi:Y\rightarrow X$ with $G$ a finite group and $\pi^*E$ a trivializable flat bundle.  We need $G$ to be finite in order for both $Rep(G)$ and $\vect_G$ to be semisimple.

Let $M=H^0_{dR}(Y,\pi^* E)$, be the space of global flat sections of $\pi^* E$.  We note that $G$ acts on $M$, and $E\simeq Y\times_G M,$ while $\Gamma(X,E) =(M\ot\Oc_Y)^G$ where $\Oc_Y$ denotes the algebra of functions on $Y$.  Furthermore, if $\Omega^i_Y$ denotes $\Omega^i\Oc_Y$ then $$\Gamma(X,E\ot_{\Oc_X}\Omega^i_X)\simeq(M\ot\Omega^i_Y)^G$$ and in fact $$H_{dR}^i(X,E)\simeq (M\ot H_{dR}^i(Y))^G.$$

\begin{remark}\rm{
Conversely, if $M$ is a $G$-representation, and $Y/X$ is a principal $G$-bundle, then $E=Y\times_G M$ is a flat bundle on $X$ to which we can equally well apply these considerations.
}\end{remark}

Let $\Cc=Rep (G)$ be the monoidal category of representations of $G$, not necessarily finite dimensional.  Recall that the center $\Zc(\Cc)$ consists of $G$-equivariant $G$-graded vector spaces, i.e., $N=\oplus N_g$ with $x\in G$ acting thus $x:N_g\rightarrow N_{xgx^{-1}}$.  Let $M$ be the $G$ representation above.  It can be viewed as an element of the center concentrated at $e\in G$.  In fact, $M$ is naturally a SAYD module.  This is not the case generally, but is specific to $Rep (G)$.  More precisely, for $V\in G$ we have ${^*V}\simeq V^*$ thus there is no difference between AYD an YD modules.  So we can form a cocyclic object $$\Hom_{Rep (G)}(M\ot\Oc_{Y}^{\ot\bullet+1}, \id)$$ which computes $HC^n_G(\Oc_Y,M)$, the type $A$ theory.  The following proposition is almost immediate after the above discussion.  The proposition itself is not new, Georgy Sharygin is to be credited with informing us of the existence of the result.

\begin{proposition}\label{cyclicderham}
With the assumptions above, we have a natural identification: $$HC^n_G(\Oc_{Y}, M)\simeq \text{ker}\,d^*|_{(E\ot_{\Oc_X}\Omega^n_X)^*}\oplus\bigoplus_{i\geq 1} H^{dR}_{n-2i}(X,E).$$
\end{proposition}
\begin{proof}
The proof consists of a chain of isomorphisms:
\begin{align*}
HC^n(\Hom_{Rep (G)}(M\ot\Oc_{Y}^{\ot\bullet+1},\id))&\simeq HC^n(\Hom(M\ot\Oc_{Y}^{\ot\bullet+1}, \id)^G)\\
&\simeq (M^*\ot HC^n(\Oc_{Y}^{\ot\bullet+1}, \id))^G\\
&\simeq (M^*\ot (\text{ker}d^*|_{(\Omega^n_Y)^*}\oplus\bigoplus_{i\geq 1}H_{n-2i}^{dR}(Y)))^G\\
&\simeq (M^*\ot \text{ker}d^*|_{(\Omega^n_Y)^*})^G\oplus\bigoplus_{i\geq 1}(M^*\ot H_{n-2i}^{dR}(Y))^G\\
&\simeq (\text{ker}(Id\ot d)^*|_{(M\ot\Omega^n_Y)^*})^G\oplus\bigoplus_{i\geq 1}(M\ot H^{n-2i}_{dR}(Y))^{*G}\\
&\simeq \text{ker}\,d^*|_{(E\ot_{\Oc_X}\Omega^n_X)^*}\oplus\bigoplus_{i\geq 1} H^{dR}_{n-2i}(X,E).
\end{align*}
\end{proof}

So we see that Hopf cyclic cohomology with coefficients may encode the de Rham cohomology of flat vector bundles under some conditions.  However this raises a natural question: Is this the only way to encode the data of a flat vector bundle $E$ on $X$?  After all, $Rep (G)$ is dual to $\vect_G$, the monoidal category of $G$-graded vector spaces with the convolution product.  These are dual via their action on $\vect$, or equivalently via the fact that they are representation categories of dual Hopf algebras: the group algebra $kG$ and the function algebra $\Oc_G$.

Let $\Cc^\vee=\vect_G$. Consider the algebra $\Oc_Y\rtimes G$ in $\Cc^\vee$ and the dual $M^*$ of the $M$ above, now considered as an SAYD for $\Cc^\vee$.  We can again form a cocyclic object $$\Hom_{\vect_G}((\Oc_Y\rtimes G)^{\ot\bullet+1}, M^*)$$ which computes $HC^{n,G}(\Oc_Y\rtimes G,M^*)$, i.e., the type $B$ theory.

The answer is given by Lemma \ref{duallemma2} and it is
$$HC^{n,G}(\Oc_Y\rtimes G,M^*)\simeq HC^n_G(\Oc_{Y}, M),$$ where the latter encodes $H^\bullet_{dR}(X,E)$ as we've seen in Proposition \ref{cyclicderham}.

\subsection{Adjoint action}\label{adjactionsection}
Here we briefly sketch out the notion of adjoint action that we feel explains the  contramodule coefficients of \cite{contra}.  Let $\Mc$ be a $\Cc$-bimodule category \cite{EGNO} and suppose that the action has right adjoints, then one may form $\Mc^{op}$, the contragradient $\Cc$-bimodule  defined via the adjoint action as follows:
\begin{align*}
\Hom_{\Mc}(c\cdot  m, m')&= \Hom_{\Mc}(m, m'\ra c),\\
\Hom_{\Mc}(m\cdot c, m')&= \Hom_{\Mc}(m, c\la m').
\end{align*}

Thus $\Mc^{op}$ as a category is the opposite category of $\Mc$ and the action is via $\la$ and $\ra$.  We often consider an expression $c\la m$ as an element of $\Mc$ in which case it is covariant in $m$ and contravariant in $c$.  Note that in the main body of this paper $(-)^{op}$ denotes simply the opposite category.  This is not the case here!

\begin{remark}
Of course if $\Mc$ is a left $\Cc$-module category, then $\Mc^{op}$ is a right $\Cc$-module category and vice versa.  Also note that if $\Mc^{op}$ exists then so does $$(\Mc^{op})^{op}=\Mc$$ where the equality is of either left, right, or $\Cc$-bimodule categories, whichever is applicable.

\end{remark}

\begin{example}\rm{
Let $\Cc$ be a \emph{rigid} monoidal category  and let $\Cc$ act on itself by $c\ot -$ and $-\ot c$. Then for $c'\in\Cc^{op}$ one has $$c'\ra c= {^*c}\ot c', \quad ~~ \text{and} \quad ~~ c\la c'=c'\ot c^*.$$ More generally, if $\Mc$ is a $\Cc$-bimodule category, then $\Mc^{op}$ is defined via $$m\ra c= {^*c}\ot m, \quad ~~ \text{and} \quad ~~ c\la m=m\ot c^*. $$  In the rigid case, even ${^{op}\Mc}$, with the natural definition, exists. }
\end{example}

\begin{example}\label{closed}\rm{
Let $H$ be a Hopf algebra over a field $k$ with an invertible antipode.
Recall that  $\hmod$, the category of left modules (not necessarily finite dimensional) over $H$ has internal Homs. For all $T, W\in \hmod$,  the left internal Hom is defined by \begin{align*}\Hom^l(W, T)&=\Hom_k(W, T),\\ h\cdot \varphi &= h\ps{1}\varphi(S(h\ps{2})-),\end{align*} and the right internal Hom is defined by \begin{align*}\Hom^r(W, T)&=\Hom_k(W, T),\\ \varphi \cdot h&= h\ps{2}\varphi(S^{-1}(h\ps{1})-).\end{align*} More precisely, we have adjunctions for $T, V, W\in \hmod$:
$$\Hom_H(V\ot W, T)=\Hom_H(V, \Hom^l(W, T)),$$ and
$$\Hom_H(W\ot V, T)= \Hom_H(V, \Hom^r(W, T)),$$ so that $$T\ra W= \Hom^r(W, T), \quad ~~ \text{and} \quad ~~ W\la T=\Hom^l(W, T)$$ for $T\in\hmod^{op}$.
}
\end{example}

The above example simply spells out that $\hmod$ is biclosed. Speaking of biclosed categories, there is a simpler example:

\begin{example}\label{commutative}\rm{
Let $R$ be a commutative algebra over $k$. Consider the monoidal category $R\,\md$ of left $R$-modules. Then we have $$\Hom_R(V\ot_R W, T)=\Hom_R(V, \Hom_R(W, T))$$ and since our category is symmetric monoidal, we get $$T\ra_R W= \Hom_R(W, T), \quad ~~ \text{and} \quad ~~ W\la_R T=\Hom_R(W, T)$$ for $T\in R\,\md^{op}$.}
\end{example}

We can, in a way, combine the two examples.  Namely, $\hmod$ comes equipped with a fiber functor, i.e., a monoidal functor (with other properties \cite{EGNO}) from $\hmod$ to $\vect$.  This functor forgets the $H$-module structure and leaves only the underlying vector space. Note that the $op$-structure above is compatible with the $op$-structure in $\vect$, as the underlying vector space of both $T'\ra W$ and $W\la T'$ is $\Hom_k(W, T)$ which is the internal $\Hom$ in $\vect$.  We can consider the case of a weak Hopf algebra $\Hc$ with a separable base $R$. Its distinguishing feature is a fiber functor not to $\vect$, but to $\bimod R$ the monoidal category of $R$-bimodules (called an $R$-fiber functor \cite{EGNO}).  Thus we have the following example of an $op$-structure compatible with the fiber functor.

\begin{example}\label{weakhopf}\rm{
Let $\Hc$ be a weak Hopf algebra with a separable base $R$. Consider the monoidal category of left $\Hc$-modules. Then we have for $W\in{_\Hc\Mc}$ and $T\in{_\Hc\Mc}^{op}$: $$T\ra W= \Hom^r_R(W, T), \quad ~~ \text{and} \quad ~~ W\la T=\Hom^l(W, T)_R$$ with the $\Hc$-action on $\Hom_R(W, T)$ and $\Hom(W, T)_R$ given as in the Example \ref{closed}.}
\end{example}

Observe that when $R=k$ we recover Hopf algebras; when $\Hc=R$ we recover the special case of the Example \ref{commutative} with the commutative $k$-algebra $\Zc(R)$, i.e., the center of the base of $\Hc$. Note that there is a big difference between $R\md$ and $\bimod R$; the former is a few copies of $\vect$ while the latter is essentially $M_n(\vect)$, i.e., the category whose objects are matrices of vector spaces, which is  $2$-Morita equivalent to $\vect$ itself.  The former is useful to keep in mind as a ``non-connected" example that clashes with the intuition developed after dealing with the usual ``connected" examples like the latter.

The point of the above is that $\Cc$ can be considered as its own bimodule category and suppose that $\Cc$ is biclosed, which for us is equivalent to $\Cc^{op}$ being defined. Let $M\in \Zc_\Cc(\Cc^{op})$, then by the results of \cite{hks} the contravariant functor $$\Hom_{\Cc}(-,M)$$ from $\Cc$ to $\vect$ is a $2$-contratrace.  In particular, for every $c\in\Cc$ we obtain an automorphism of $\Hom_{\Cc}(c,M)$ via: $$\Hom_{\Cc}(c,M)\simeq\Hom_{\Cc}(1,M\ra c)\simeq\Hom_{\Cc}(1,c\la M)\simeq\Hom_{\Cc}(c,M).$$  It is immediate that this automorphism is functorial in $c$ and thus by a Yoneda argument is induced by an automorphism  $\sigma_M\in\Hom_{\Cc}(M,M)$ (compare with the map $u_M$ in \cite{js}).  Let us denote by $\Zc'_\Cc(\Cc^{op})$ those objects $M$ in $\Zc_\Cc(\Cc^{op})$ with $\sigma_M=Id$.
Thus if $M\in \Zc'_\Cc(\Cc^{op})$ then $\Hom_{\Cc}(-,M)$ is a symmetric $2$-contratrace;  for the case of $\Cc=\hmod$ this leads to cyclic cohomology with contramodule coefficients.

\begin{remark}\em{
The above observations can be summarized thus: If $\Cc$ is biclosed, then we have a fully faithful embedding $\Zc'_\Cc(\Cc^{op})\rightarrow\Zc'_\Cc(\Cc^*)$ and though there is no natural analogue of $S^2$ of the Hopf case in general, and so no notion of $SAYD(\Cc)$, there is an analogue of SAYD contra-modules, and it is $\Zc'_\Cc(\Cc^{op})$.  Furthermore, if a left exact $\Fc\in \Zc'_\Cc(\Cc^*)$ is representable, then it is representable by a SAYD contramodule, so that the image of   $\Zc'_\Cc(\Cc^{op})$ in $\Zc'_\Cc(\Cc^*)$ consists of representable (automatically left exact) symmetric $2$-contratraces.

}\end{remark}

In addition to the above, we point out that Theorem \ref{sl} can be restated in terms of contramodules.  Let us assume that both $\Cc$ and $\Dc$ are biclosed.  If $F:\Cc\to\Dc$ is a monoidal functor with an exact right adjoint ${^*F}$, then ${^*F}:\Zc'_\Dc(\Dc^{op})\to \Zc'_\Cc(\Cc^{op})$.  Furthermore, if $A$ is an algebra in $\Dc$ and $M\in \Zc'_\Dc(\Dc^{op})$ then $M\ra A\in\Zc'_{\bimod_\Dc A}(\bimod_\Dc A)^{op}$.  Therefore, if $\Nc$ is an admissible $(\Cc,\Dc)$-bimodule then we have $\Nc^*:\Zc'_\Dc(\Dc^{op})\to \Zc'_\Cc(\Cc^{op})$ and the rest of the Theorem \ref{sl} applies.

\begin{remark}\em{
As mentioned above, there is no natural analogue of $S^2$ that yields the monoidal autoequivalence of $\hmod$ and allows the definition of the bimodule category ${^\#\hmod}$, and thus of $\Zc'_{\hmod}({^\#\hmod})$, where the latter are exactly ${_H SAYD^H}$. However, if $\Cc$ is not only biclosed, but the equation $$c\la\id\simeq\id\ra c^\#$$ has a solution defining the autoequivalence $(-)^\#$, then $SAYD(\Cc)$ should be exactly $\Zc'_\Cc(^\#{\Cc})$. In the case when $\Cc$ is rigid, we have $(-)^\#=(-)^{**}$ is a solution.  The functor from SAYD modules to SAYD contramodules is as expected, namely $M\mapsto\id\ra M$.

}\end{remark}

%%%%%%%%%%%%%%%%%%%%%%%%%%%%%%%%%%%%%%%%%%%%%%%%%%%%%%%%%%%%%%%%%%%%%%%%%%%%%%%%%%%%%%%%%%%%%%%%%%%%%%

\bigskip

\noindent Department of Mathematics and Statistics,
University of Windsor, 401 Sunset Avenue, Windsor, Ontario N9B 3P4, Canada

\noindent\emph{E-mail address}:
\textbf{ishapiro@uwindsor.ca}


\begin{thebibliography}{9}

\bibitem[BS]{catdual} G. B\"{o}hm, D. Stefan, \emph{
A categorical approach to cyclic duality.}
J. Noncommut. Geom. 6 (2012), no. 3, 481--538.

\bibitem[B]{contra} T. Brzezinski, \emph{Hopf-cyclic homology with contramodule coefficients},  Quantum groups and noncommutative spaces, 1–8,
Aspects Math., E41, Vieweg + Teubner, Wiesbaden, 2011.

\bibitem[CM]{cm}  A. Connes, H. Moscovici, \emph{Cyclic cohomology
and Hopf algebras}, Lett. Math. Phys. 48 (1999), 97--108.



\bibitem[CV]{catex} G. Corti\~{n}as, C. Valqui, \emph{
Excision in bivariant periodic cyclic cohomology: a categorical approach.}
Special issue in honor of Hyman Bass on his seventieth birthday. Part II.
K-Theory 30 (2003), no. 2, 167--201.


\bibitem[EGNO]{EGNO} P. Etingof, Sh. Gelaki, D. Nikshych, and V. Ostrik, \emph{Tensor Categories},  Mathematical Surveys and Monographs, 205. American Mathematical Society, Providence, RI, 2015.

\bibitem[HKS]{hks} M. Hassanzadeh, M. Khalkhali, and I. Shapiro, \emph{Monoidal Categories,  2-Traces, and Cyclic Cohomology}, preprint  arXiv:1602.05441.

\bibitem[HKRS]{hkrs} P. M. Hajac, M. Khalkhali, B. Rangipour, and Y. Sommerh¨auser, \emph{Hopf-cyclic homology and cohomology with coefficients.} C. R. Math. Acad. Sci.Paris 338 (2004), no. 9, 667--672.

\bibitem[JS]{js} P. Jara, D. Stefan, \emph{Hopf-cyclic homology and relative cyclic homology of Hopf-Galois extensions.}
Proc. London Math. Soc. (3) 93 (2006), no. 1, 138--174.

\bibitem[K]{catcoeff} D. Kaledin, \emph{Cyclic homology with coefficients.} Algebra, arithmetic, and geometry: in honor of Yu. I. Manin. Vol. II, 23--47, Progr. Math., 270, Birkhäuser Boston, Inc., Boston, MA, 2009.

\bibitem[KKS]{catadj} N. Kowalzig, U. Kr\"{a}hmer, P. Slevin, \emph{Cyclic homology arising from adjunctions.}
Theory Appl. Categ. 30 (2015), 1067--1095.

\bibitem[Lo]{loday} J. L. Loday, \emph{Cyclic homology}, Grundlehren der mathematischen Wissenschaften Vol. 301, Springer,  (1998).

\bibitem[RS]{ser} B. Rangipour, S. S\"{u}tl\"{u}, \emph{Characteristic classes of foliations via SAYD-twisted cocycles}, J. Noncommut. Geom. 9 (2015), no. 3, 965--998.

\bibitem[S]{catmackey} J. S{\l}omi\'{n}ska, \emph{
Noncommutative Mackey functors and Hopf-cyclic homology.}
K-Theory 37 (2006), no. 4, 379--394.


\bibitem[W]{weibel} C. Weibel, \emph{An introduction to homological algebra.} Cambridge Studies in Advanced Mathematics, 38. Cambridge University Press, Cambridge, (1994).







\end{thebibliography}
\end{document}